\documentclass[12pt]{amsart}
\usepackage{url}
\usepackage{amssymb}
\usepackage[margin=2cm]{geometry}
\usepackage{verbatim}
\usepackage{hyperref}
\newtheorem{theorem}{Theorem}

\newtheorem{lemma}{Lemma}

\newcommand{\tmop}[1]{\ensuremath{\operatorname{#1}}}

\def\sumstar{\sideset{}{^*}\sum}

\numberwithin{equation}{section}
\numberwithin{lemma}{section}
\numberwithin{proposition}{section}
\numberwithin{corollary}{section}

\title[The variance of the number of lattice points in narrow sectors]{The variance of the number of lattice points in narrow sectors}

\author{Stephen Lester}
\address{Department of Mathematics, King's College London, London WC2R 2LS, UK}
\email{steve.lester@kcl.ac.uk}

\author{Ezra Waxman}
\address{Unit of Mathematics, Afeka -- The Academic College of Engineering in Tel Aviv, Mivtsa Kadesh St 38, Tel Aviv-Yafo 6998812, Israel}
\email{ezraw@afeka.ac.il}

\author{Nadav Yesha}
\address{Department of Mathematics, University of Haifa, 3103301 Haifa, Israel}
\email{nyesha@univ.haifa.ac.il}

\begin{document}

\begin{abstract}
We study the variance of the number of lattice points in smoothed narrow sectors, where the average is taken over the position of the sectors. There are three regimes corresponding to specific ranges for the open angle of the sectors and in each we obtain an asymptotic formula for the variance.
\end{abstract}

\maketitle

\section{Introduction}
\subsection{Main results}
In this article we study the variance of the number of lattice points lying in narrow sectors, where the variance is taken over the position of the sector. We shall consider the smoothed variance.

Let $W,f$ be real-valued smooth functions with rapid decay. Further assume $W$ is compactly supported on the positive reals and $f$ is even. Let $R,K$ be parameters. Here $R$ describes the radius of the circle and $\frac{\pi}{2K}$ the open angle. Throughout we use the $\varepsilon$-convention that $\varepsilon>0$ denotes a small, arbitrary positive real number that may vary from line-to-line. The constants $c_{1},c_{2},c_{3},c_{4},c_{5},c_{6},C_{0},D_{0}$,
and $C_{1}$ (depending only on $W,f$) are defined in Appendix \ref{sec:Constants} and will
be used throughout. Define the smooth counting function around $\theta \in \mathbb R/(\frac{\pi}{2} \mathbb Z)$ by
\[
F_K(\theta)=\sum_{j \in \mathbb Z}
f\left(\frac{K}{\pi/2}\left(\theta+ j \cdot \frac{\pi}{2} \right) \right)
=\frac{1}{K} \sum_{k \in \mathbb Z} \widehat f\left( \frac{k}{K}\right) e^{4i\theta k},
\]
where the last equality follows from the Poisson summation formula and the Fourier transform of $g \in L^1(\mathbb R)$ is
\[
\widehat g(\xi) =\int_{\mathbb R} g(t) e(-\xi \cdot t) \, dt, \qquad e(x)=e^{2\pi ix}.
\]
For $\alpha \in \mathbb Z^2 \setminus \{0\}$ write $\theta_{\alpha}$ for the angle between $\alpha$ and the positive $x$-axis, which is understood modulo $2\pi$.

The expected value of the lattice point count over the position of the sector is
\[
\begin{split}
\int_0^{\pi/2}
&\sum_{\alpha \in \mathbb Z^2\setminus \{0\}} F_K(\theta-\theta_{\alpha}) W\left(\frac{\lVert \alpha \rVert^2}{R^2}\right) \frac{d\theta}{\pi/2}\\
&=  \frac{1}{K} \sum_{\alpha \in \mathbb Z^2\setminus \{0\}} W\left(\frac{\lVert \alpha \rVert^2}{R^2}\right) \sum_{k\in \mathbb Z} e^{-4ik\theta_{\alpha}} \widehat f\left(\frac{k}{K} \right) \int_0^{\pi/2}
 e^{4ik\theta} \frac{d\theta}{\pi/2}=\frac{\widehat f(0)}{K} \sum_{\alpha \in \mathbb Z^2\setminus \{0\}} W\left( \frac{\lVert \alpha \rVert^2}{R^2}\right).
 \end{split}
\]
Let us consider the variance
\begin{equation}
\label{eq:Var_Def}
\tmop{Var}(R,K)=\int_0^{\pi/2} \left| \sum_{\alpha \in \mathbb Z^2\setminus \{0\}} F_K(\theta-\theta_{\alpha}) W\left(\frac{\lVert \alpha \rVert^2}{R^2}\right)-\frac{\widehat f(0)}{K} \sum_{\alpha \in \mathbb Z^2\setminus \{0\}} W\left( \frac{\lVert \alpha \rVert^2}{R^2}\right) \right|^2 \frac{d\theta}{\pi/2}.
\end{equation}

The  main results of this article are the following asymptotic formulas for $\tmop{Var}(R,K)$ in the various regimes: slowly shrinking (or non-shrinking) sectors, intermediately narrow sectors, and very narrow sectors.

\begin{theorem} \label{thm:mainthm}
Let $\eta>0$, and let $C_{1},C_{0},D_{0}$ be the constants defined in (\ref{eq:main_consts}).
\begin{enumerate}
\item For $1 \le K<R^{1-\eta}$ and any $A>0$, we have that
\[
\textup{Var}\left(R,K\right)=O\left(R^{-A}\right).
\]
\item For $R^{1+\eta}< K < R^{2-\eta}$, and for any $\varepsilon>0$, we have that
\[
\tmop{Var}(R,K)= \frac{R^2}{K} \bigg(  C_1 \log \frac{K}{R} +C_0 +O\left( \frac{R^{1+\varepsilon}}{K}\right) \bigg).
\]
\item For $K>R^{2+\eta}$, we have that
\[
\textup{Var}\left(R,K\right)=\frac{R^{2}}{K} \bigg( C_{1} \log R+D_{0}+O\left(\frac{R^2}{K}+\frac{1}{R}\right)\bigg).
\]
\end{enumerate}

\end{theorem}

\subsection{Related results}

The asymptotic growth of the number of lattice points in narrow
sectors with a fixed position $\theta$ was studied by the second and the third authors in \cite{WaxmanYesha}.
The results depend both on the growth rate of $K$ as well as on
the rationality/irrationality type of $\tan\theta$. In particular,
if $\tan\theta$ is a Diophantine irrational (i.e., of irrationality exponent $2$), then the number of lattice points is asymptotic to the area
of the sector, as long as $K\ll R^{2-\eta}$ for some $\eta>0$.

Our analysis in the current paper does not address the behavior at the transition scales
$K\approx R$ and $K\approx R^{2}$. The scale $K = R^2$ (the local scale) is particularly
natural when considering statistics of sequences. While the variance for the number of lattice
points itself has not been directly addressed at this scale, the existence of
a (non-Poissonian) limiting distribution was established by Marklof and
Strömbergsson  \cite{MarklofStrombergsson}. The convergence of
the moments at the local scale (and particularly the variance) was further established
by El-Baz, Marklof and Vinogradov \cite{EMV} for affine lattices
$\mathbb{Z}^{2}+\xi$, when $\xi$ satisfies a certain Diophantine
condition. A closely related question concerns the distribution of
the number of visible points in sectors, which in turn is connected
to the distribution of Farey fractions, for which the convergence of all
moments at the local scale was obtained by Boca and Zaharescu \cite{BocaZaharescu}. The first named author and Wigman \cite{LesterWigman2024} as well as Bordignon and Kurlberg \cite{BordignonKurlberg2026} studied the distribution of lattice points near the circle. The former work obtains an asymptotic formula for the variance of lattice points lying in the annulus with inner radius $R-1/\sqrt{2}$ and outer radius $R+1/\sqrt{2}$ over sectors of open angle $2\pi/R$, and the latter work establishes results on the correlations of lattice points near the circle at the local scale.

Another closely related problem concerns the number of Gaussian primes
in random narrow sectors. The variance was studied by Rudnick and
the second author \cite{RudnickWaxman}, who proved an asymptotic
formula when $R^{2}=o\left(K\right)$, i.e., for very narrow sectors,
and gave a conjecture for the asymptotics when $K<R^{2-\eta}.$ In \cite{ChenEtAl},
the authors extended the conjecture to include a lower order term
(as established in our paper), and showed how
to derive their conjecture from the recipe of the ratios conjecture \cite{ConreyFarmerZirnbauer} applied
to the Hecke $L$-functions.

Our work may also be viewed as providing the variance of sums of an arithmetic function over a collection of sparse sets. There is a large literature on this topic including work on the variance of sums of divisor functions and Hecke eigenvalues over sparse sets such as short intervals \cite{Jutila1984,Ivic2009,Lester2016,Jaasaari2020,BettinConrey2021}, arithmetic progressions \cite{Blomer2008,Lu2009, LauZhao2012,FGKM2014, HarperSoundararajan2017, RodgersSoundararajan2018, Nguyen2024}, and the initial terms \cite{Carmichael2026I,Carmichael2026II}. Furthermore, the
variance and distribution of lattice points in thin annuli have been studied in \cite{BleherLebowitz1995,HughesRudnick2004,Wigman2006}.

\subsection{Initial manipulations}

Writing
\[
S(k)=\sum_{\alpha \in \mathbb Z^2\setminus \{0\}} e^{-4i\theta_{\alpha} k} W\left(\frac{\lVert \alpha \rVert^2}{R^2}\right),
\]
we see that
\begin{equation} \label{eq:var-id}
\begin{split}
\tmop{Var}(R,K)&=\frac{1}{K^2} \sum_{k_1,k_2 \in \mathbb Z \setminus \{0\}} \widehat f\left( \frac{k_1}{K}\right) \overline {\widehat f\left( \frac{k_2}{K}\right)} S(k_1) \overline{S(k_2)} \int_0^{\pi/2} e^{4i\theta(k_1-k_2)} \frac{d\theta}{\pi/2}\\
&= \frac{1}{K^2} \sum_{k \in \mathbb Z \setminus \{0\}} \left|\widehat f\left( \frac{k}{K}\right)\right|^2 |S(k)|^2.
\end{split}
\end{equation}

To analyze $S(k)$ when $K<R^{2-\eta}$ we apply the Poisson summation formula.
Given a smooth rapidly decaying function $g$ supported on the positive reals and $\ell$ an integer, define for $\xi \ge 0$
\[
\mathcal B_{\ell}(g)(\xi) = \int_0^{\infty} g(y) J_{\ell}(2\pi \sqrt{\xi y}) \, dy,
\]
where $ J_{\ell}$ is the Bessel function of the first kind.
Repeatedly integrating by parts, one can show that
\begin{equation} \label{eq:IBP}
\mathcal B_{\ell}(g)(\xi) \ll_g \frac{1+|\ell|^A}{1+\xi^{A/2}},
\end{equation}
for any $A>0$ (cf. \cite[Lemma 9.3]{LesterWigman2024}, \cite[Lemma 2.4]{Carmichael2026I}).
Let us define
\[
\lambda_{4k}(n)=\sum_{\substack{ \alpha \in \mathbb Z^2 \\ \lVert \alpha \rVert^2 =n }} e^{4i \theta_{\alpha} k}.
\]
The Poisson summation formula, see \cite[Lemma 9.2]{LesterWigman2024}, yields the following result, for $k \neq 0$,
\[
S(k)=\sum_{n \ge 1} \lambda_{4k}(n) W\left(\frac{n}{R^2} \right) = \pi R^2 \sum_{n \ge 1} \lambda_{4k}(n) \mathcal B_{4k}(W)(nR^2).
\]
Hence, we obtain
\begin{equation} \label{eq:poissonapplied}
\tmop{Var}(R,K)= \frac{\pi^2 R^4}{K^2}  \sum_{k \in \mathbb Z \setminus \{0\}}  \left|\widehat f\left( \frac{k}{K}\right)\right|^2
\left| \sum_{n \ge 1} \lambda_{4k}(n) \mathcal B_{4k}(W)(nR^2) \right|^2.
\end{equation}

The last identity already implies the claimed bound for the variance in the regime $1\le K<R^{1-\eta}$, that is, for slowly shrinking (or non-shrinking) sectors.

\begin{proof}[Proof of Theorem \ref{thm:mainthm}, part (1)]

Assume that $1\le K<R^{1-\eta}$. The estimate $\textup{Var}\left(R,K\right)=O\left(R^{-A}\right)$
then follows by \eqref{eq:poissonapplied}, \eqref{eq:IBP}, and the rapid decay of $\widehat{f}$.
\end{proof}

\subsection{Discussion of the proof and organization of the paper}
The bulk of the paper is dedicated to proving the second part of Theorem \ref{thm:mainthm}, which concerns the intermediate regime $R^{1+\eta}< K < R^{2-\eta}$.  To compute $\tmop{Var}(R,K)$ we start with the identity \eqref{eq:poissonapplied} in which we have already expanded the condition that our lattice points lie in a narrow sector by a Fourier series, integrated over the position of the sector using Parseval's identity, and then transformed the sum over lattice points through an application of the Poisson summation formula. This provides an exact formula for $\tmop{Var}(R,K)$ and by \eqref{eq:IBP} we are left with estimating, essentially, the following second moment
\[
\sum_{1\le|k|\le K} \bigg|\sum_{1\le n \le \frac{K^{2}}{R^2}} \lambda_{4k}(n) J_{4k}(2\pi R \sqrt{n}) \bigg|^2.
\]
Since $K<R^{2-\eta}$ the length of the $n$-sum is $K^2/R^2<K$, which is shorter than the length of the $k$-sum which has length $2K$. This suggests that the main contribution will come from a diagonal term consisting of pairs of lattice points lying on the same line through the origin.  However, the relevant Bessel functions, $J_{4k}(2\pi R\sqrt{n})$, occur at the transition regime where their arguments, $2\pi R\sqrt{n}$, are approximately the same size as their indices, $4k$, which complicates the analysis.

To overcome this difficulty, instead of using asymptotic expansions of the Bessel functions, our strategy starts with the following identity
\begin{equation}
\label{eq:Bessel}
J_{k}(2\pi x)=\int_{-1/2}^{1/2}  e(-x \sin (2\pi t)) e(kt) \, dt.
\end{equation}
After using this representation we then execute the sum over $k$ by another application of the Poisson summation formula. This yields an exact identity in terms of an oscillatory integral, which is established in Section \ref{sec:summing}. Localizing this integral, in Section \ref{sec:collinear} we show that the main contribution arises from pairs of $\mathbb Z^2$-lattice points $\alpha, \beta$ with relative angle at most $o( \frac{1}{\lVert \alpha \rVert \lVert \beta \rVert})$ and since these are $\mathbb Z^2$-lattice points integrality forces them to lie on the same line through the origin, up to rotation by an appropriate multiple of $\pi/2$. This allows us to express these $\mathbb Z^2$-lattice points $\alpha, \beta$ as $\alpha =m \delta$, $\beta=n \delta$ (up to a rotation by a multiple of $\pi /2$) where $m,n$
 are integers and $\delta$ is a primitive, non-zero $\mathbb Z^2$-lattice point.

 We are left with an arithmetic sum over primitive lattice points $\delta$ and integers $m,n$, which we estimate in Section \ref{sec:mainterm} using Mellin inversion and a contour integration argument. In the proof of Lemma \ref{lem:mainterm} we express the primitive lattice point sum in terms of the Epstein zeta-function and the $m,n$ sum by another Dirichlet series, which we relate to the Riemann zeta-function. Each Dirichlet series has a simple pole at $s=1$, which leads to a double pole that accounts for the logarithm appearing in the main term of our theorem.

 In Section \ref{sec:proof} we combine the estimates of the preceding sections to complete the proof of the second part of Theorem \ref{thm:mainthm}.


Finally, we prove in Section \ref{sec:VeryNarrowSec} the third part of Theorem \ref{thm:mainthm}. Here, the idea is that after expanding the square in the definition \eqref{eq:Var_Def} of the variance, the condition $K>R^{2+\eta}$ implies that up to a negligible error term, the entire contribution for the second moment comes from the diagonal terms where $\theta_{\alpha} = \theta_{\beta}$. Applying the Mellin inversion theorem, we express the sum in question as a double contour integral, which is then evaluated by shifting the contours and collecting the residues.

\section{The intermediate regime, preparatory results}
\subsection{Summing over frequencies} \label{sec:summing}
 Expanding the square in \eqref{eq:poissonapplied} we see that we need to estimate a quantity of the form
\[
\sum_{k \in \mathbb Z} g\left( \frac{k}{K}\right) e^{4ik( \theta_{\alpha}-\theta_{\beta})} J_{4k}(2\pi R\lVert \alpha \rVert \sqrt{v_1})J_{4k}(2 \pi R\lVert \beta \rVert \sqrt{v_2}),
\]
where $g$ is a smooth, rapidly decaying function. An exact identity for this quantity is established in the following lemma, which is closely related to \cite[Lemma 9.6]{LesterWigman2024} and \cite[Lemma 5.8]{Iwaniec1997}.

\begin{lemma} \label{lem:bessel-id} Let $x,y \in \mathbb R$ and $g$ satisfy $g(u), \widehat g(u) \ll (1+|u|)^{-1-\delta}$ for some $\delta>0$. Let $K>0$ and $\theta \pmod{ \tfrac{1}{4}}$. Then
\begin{equation}
\begin{split}
&\sum_{k \in \mathbb Z} e(4k \theta) J_{4k}(2\pi x) J_{4k}(2\pi y) g \bigg( \frac{k}{K} \bigg)\\
&=\frac{1}{4}  \sum_{0\le b<4}
\int_{-1/2}^{1/2} \int_{\mathbb R} e(-x \sin(2\pi t_2)+y \sin(2\pi(\tfrac{t_1}{4K}+t_2+\theta+\tfrac{b}{4}))) \widehat g(t_1) \, dt_1 dt_2.
\end{split}
\end{equation}

\end{lemma}

\begin{proof}
Starting with the well-known formula \eqref{eq:Bessel}
with $k \in \mathbb Z$
we see that
\[
\begin{split}
&\sum_{k \in \mathbb Z} e(4k \theta) J_{4k}(2\pi x) J_{4k}(2\pi y) g \bigg( \frac{k}{K} \bigg) \\
&=\frac{1}{4} \sum_{0 \le b < 4}  \int_{-1/2}^{1/2} \int_{-1/2}^{1/2} e(-x \sin(2\pi t_2)-y\sin(2\pi t_1)) \sum_{k \in \mathbb Z} e(k(t_1+t_2+\theta+\tfrac{b}{4} )) g \bigg( \frac{k}{4K}\bigg) dt_1 dt_2 \\
&= K \sum_{0 \le b < 4}  \int_{-1/2}^{1/2} \int_{-1/2}^{1/2} e(-x \sin(2\pi t_2)-y\sin(2\pi t_1)) \sum_{j \in \mathbb Z} \widehat g(4K(j-(t_1+t_2+\theta+\tfrac{b}{4}))) dt_1 dt_2.
\end{split}
\]
In the integral with respect to $t_1$ make the substitution $u=j-t_1-t_2-\theta-\frac{b}{4}$ to obtain 
\[
\begin{split}
 &\sum_{j \in \mathbb Z} \int_{-1/2}^{1/2} e(-y \sin(2\pi t_1)) \widehat g(4K(j-(t_1+t_2+\theta+\tfrac{b}{4}))) dt_1\\
 &\qquad \qquad \qquad=
  \sum_{j \in \mathbb Z} \int_{-1/2+j-t_2-\theta-\frac{b}{4}}^{1/2+j-t_2-\theta-\frac{b}{4}} e(y \sin(2\pi (u+t_2+\theta+\tfrac{b}{4})) \widehat g(4Ku) du \\
  &\qquad \qquad \qquad=\int_{\mathbb R} e(y \sin(2\pi (u+t_2+\theta+\tfrac{b}{4})) \widehat g(4Ku) du.
\end{split}
\]
Making the change of variables $t_1=4Ku$ completes the proof.
\end{proof}


\subsection{Reducing to collinear lattice points} \label{sec:collinear}

For positive real numbers $v_1,v_2$, let
\begin{equation} \label{eq:sigmadef}
\begin{split}
\Sigma_{K,R}(v_1,v_2)
={}&
\sum_{\substack{m\in\mathbb N,\ n\in\mathbb Z\\
1\le m\le |n|\le \frac{K^{1+\varepsilon}}{R}}}
\left(2-1_{m=|n|}\right)
\sumstar_{\substack{\delta\in\mathbb Z^2\\
1\le \lVert\delta\rVert\le \frac{K^{1+\varepsilon}}{|n|R}}} \,
\int\limits_{\substack{t \in[-1/2,1/2] \\ \lVert t \rVert_{\frac12 \mathbb Z} \le \frac{1}{R^{1-\varepsilon}m \lVert \delta \rVert}}}
e\left(R\lVert\delta\rVert
\left(n\sqrt{v_2}-m\sqrt{v_1}\right)\sin(2\pi t)\right)
\\
&\hspace{25mm}\times
|\widehat f|^2\left(
\frac{\pi nR\lVert\delta\rVert\sqrt{v_2}}{2K}
\cos(2\pi t)\right)\,dt,
\end{split}
\end{equation}
where $\lVert\cdot\rVert_{\frac12\mathbb Z}$ denotes the distance
to the nearest half-integer, $1_{m=|n|}$ is one if $m=|n|$ and otherwise is zero, and $\sumstar$ designates that the sum is over primitive lattice points.
We prove the following lemma.

\begin{lemma} \label{lem:poissonagain}
Let $\eta>0$.
For $R^{1+\eta}< K< R^{2-\eta}$, we have for any $\varepsilon>0$ that
\[
\tmop{Var}(R,K)
=
\frac{\pi^2R^4}{K^2}
\int_{\mathbb R^2}W(v_1)W(v_2)
\Sigma_{K,R}(v_1,v_2)\,dv_1\,dv_2
+O\left(\frac{R^{2+\varepsilon}}{K^2}\right).
\]
\end{lemma}

\begin{proof}
By \eqref{eq:IBP}, \eqref{eq:poissonapplied}, the rapid decay of $\widehat f$, and the compact
support of $W$, we see that
\[
\begin{split}
\tmop{Var}(R,K)
={}&
\frac{\pi^2R^4}{K^2}
\sum_{k\in\mathbb Z}
\left|\widehat f\left(\frac{k}{K}\right)\right|^2
\left|
\sum_{1\le n\le K^{2+\varepsilon}/R^2}
\lambda_{4k}(n)\mathcal B_{4k}(W)(nR^2)
\right|^2
+O(R^{-A})
\end{split}
\]
for any $A>0$. Here we have added the $k=0$ term to the right-hand side, which is harmless, as it is easily seen to be $O(R^{-A})$ by \eqref{eq:IBP}. Applying Lemma \ref{lem:bessel-id} with
\[
x=R\lVert\alpha\rVert\sqrt{v_1},
\qquad
 y=R\lVert\beta\rVert\sqrt{v_2},
\qquad
2\pi\theta=\theta_\alpha-\theta_\beta,
\]
and $g=|\widehat f|^2$ we obtain
\begin{equation} \label{eq:B}
\begin{split}
&\tmop{Var}(R,K)
={}
\frac{\pi^2R^4}{4K^2}
\int_{\mathbb R^2}W(v_1)W(v_2)
\sum_{\substack{\alpha,\beta\in\mathbb Z^2\\
1\le \lVert\alpha\rVert,\lVert\beta\rVert
\le K^{1+\varepsilon}/R}}
\sum_{0\le b<4}
\int_{-1/2}^{1/2}\int_{\mathbb R}(f\ast f)(t_1)
\\
&
\times e\left(
-x\sin(2\pi t_2)
+y\sin\left(2\pi\left(
\frac{t_1}{4K}+t_2+\theta+\frac b4
\right)\right)\right)
\,dt_1\,dt_2\,dv_1\,dv_2
+O(R^{-A})
\end{split}
\end{equation}
for any $A>0$.

Integrating by parts repeatedly, we see that
\begin{equation}\label{eq:firstIBP-new}
\int_0^\infty W(v_1)
e\left(-x\sin(2\pi t_2)\right)\,dv_1
\ll_A
\left(1+R\lVert\alpha\rVert|\sin(2\pi t_2)|\right)^{-A}
\end{equation}
for any $A>0$. Similarly,
\begin{equation}\label{eq:secondIBP-new}
\begin{split}
&\int_0^\infty W(v_2)
e\left(y\sin\left(2\pi\left(
\frac{t_1}{4K}+t_2+\theta+\frac b4
\right)\right)\right)\,dv_2
\\
&\qquad\ll_A
\left(1+R\lVert\beta\rVert
\left|\sin\left(2\pi\left(
\frac{t_1}{4K}+t_2+\theta+\frac b4
\right)\right)\right|\right)^{-A}
\end{split}
\end{equation}
for any $A>0$.

Up to a negligible error of size $O(R^{-A})$, by
\eqref{eq:firstIBP-new} we may restrict to
\begin{equation}\label{eq:D}
t_2\in[-1/2,1/2],
\qquad
\lVert t_2\rVert_{\frac12\mathbb Z}
\le \frac{1}{R^{1-\varepsilon}\lVert\alpha\rVert}.
\end{equation}
By \eqref{eq:secondIBP-new}, we may
also restrict to
\[
\left\lVert
\frac{t_1}{4K}+t_2+\theta+\frac b4
\right\rVert_{\frac12\mathbb Z}
\le \frac{1}{R^{1-\varepsilon}\lVert\beta\rVert},
\]
also up to a negligible error of size $O(R^{-A})$. Additionally,
$|t_1|\le R^\varepsilon$ up to a negligible error.

By symmetry, it suffices to consider $\lVert \alpha \lVert \le \lVert \beta \rVert$ and weight the terms $\lVert \alpha \lVert \neq \lVert \beta \rVert$ by a factor of $2$. Hence, since
\[
R^{1-\varepsilon}\lVert\alpha\rVert\le K,
\qquad
\lVert\alpha\rVert\le \lVert\beta\rVert,
\]
we conclude that
\begin{equation} \label{eq:sizebd}
\left\lVert\theta+\frac b4\right\rVert_{\frac12\mathbb Z}
\ll \frac{1}{R^{1-\varepsilon}\lVert\alpha\rVert}
=o\left(\frac{1}{\lVert\alpha\rVert\lVert\beta\rVert}\right),
\end{equation}
where in the last deduction we also used that
\[
\lVert\beta\rVert\le \frac{K^{1+\varepsilon}}{R}
\le R^{1-\eta+\varepsilon}
\]
since $K<R^{2-\eta}$.

Write $\beta_b$ for the clockwise rotation of $\beta$
through the angle $b\pi/2$. Also write
\[
\alpha=(a_1,a_2),
\qquad
\beta_b=(c_{1,b},c_{2,b}).
\]
Noting
\[
2\pi\left(\theta+\frac b4\right)
=\theta_\alpha-\theta_\beta+\frac{b\pi}{2}
=\theta_\alpha-\theta_{\beta_b},
\]
we will now infer that
\[
\theta_\alpha\equiv\theta_{\beta_b}\pmod{\pi}.
\]
To see this, observe that if
\[
\theta_\alpha\not\equiv\theta_{\beta_b}\pmod{\pi},
\]
then
\begin{equation}
\label{angles_ineq}
\left|\sin(\theta_\alpha-\theta_{\beta_b})\right|
=
\frac{|a_1c_{2,b}-a_2c_{1,b}|}
{\lVert\alpha\rVert\lVert\beta\rVert}
\ge
\frac{1}{\lVert\alpha\rVert\lVert\beta\rVert},
\end{equation}
which contradicts \eqref{eq:sizebd}. Here we have used that $a_1c_{2,b}-a_2c_{1,b}$ is an integer and if it is non-zero it must be $\ge 1$ in absolute value.

Since $\theta_\alpha\equiv\theta_{\beta_b}\pmod{\pi}$, we have
\[
\theta_\alpha\equiv \theta_{\beta_b} \text{, or, } \theta_{\beta_b} +\pi\pmod{2\pi}
\]
and we may write
\begin{equation} \label{eq:C}
\alpha=m\delta,
\qquad
\beta_b=n\delta,
\end{equation}
where $m\in\mathbb N$, $n\in\mathbb Z\setminus\{0\}$, and
$\delta=(d_1,d_2)$ with $\gcd(d_1,d_2)=1$. Here the sign of $n$ records whether $\beta_b$ or $-\beta_b$ has the same angle as $\alpha$. Moreover,
\[
1\le m\le |n|\le \frac{K^{1+\varepsilon}}{R}.
\]

Since
\[
2\pi\left(\theta+\frac b4\right)\equiv0\pmod{\pi},
\]
we have that
\begin{equation} \label{eq:trig}
\begin{split}
\sin\left(2\pi\left(
\frac{t_1}{4K}+t_2+\theta+\frac b4
\right)\right)=
\sin\left(2\pi\left(\frac{t_1}{4K}+t_2\right)\right)
\cos\left(2\pi\left(\theta+\frac b4\right)\right).
\end{split}
\end{equation}
Also, since
\[
\beta_b=n \delta
\]
we have that
\begin{equation} \label{eq:simple-id}
\lVert\beta\rVert
\cos\left(2\pi\left(\theta+\frac b4\right)\right)
=n\lVert\delta\rVert.
\end{equation}
The right-hand side of \eqref{eq:simple-id} does not depend on $b$ and for each $(m,n,\delta)$ there are four corresponding tuples $(\alpha,\beta,b)$ so that the factor of $1/4$ appearing on the right-hand side of \eqref{eq:B} cancels after making this substitution.

Using \eqref{eq:trig}, \eqref{eq:simple-id}, and a Taylor expansion, we see that
\begin{equation} \label{eq:colinear}
\begin{split}
&\lVert\beta\rVert R\sqrt{v_2}
\sin\left(2\pi\left(
\frac{t_1}{4K}+t_2+\theta+\frac b4
\right)\right)
\\
&\qquad=
 n\lVert\delta\rVert R\sqrt{v_2}
\sin\left(2\pi\left(\frac{t_1}{4K}+t_2\right)\right)
\\
&\qquad=
 n\lVert\delta\rVert R\sqrt{v_2}\sin(2\pi t_2)
+
\frac{\pi n\lVert\delta\rVert R\sqrt{v_2}}{2K}
 t_1\cos(2\pi t_2)
+
O\left(\frac{|n|\lVert\delta\rVert  R^{1+\varepsilon}}{K^2}\right).
\end{split}
\end{equation}
Using \eqref{eq:colinear} and that $|n|\lVert \delta \rVert <K^{1+\varepsilon}/R$ we obtain
\begin{equation} \label{eq:taylorapplied}
\begin{split}
&\int_{\mathbb R}
e\left(y
\sin\left(2\pi\left(
\frac{t_1}{4K}+t_2+\theta+\frac b4
\right)\right)\right)
(f \ast f)(t_1)\,dt_1
\\
&\qquad=
e\left(n\lVert\delta\rVert R\sqrt{v_2}
\sin(2\pi t_2)\right)
 |\widehat f|^2 \left(\frac{\pi n\lVert\delta\rVert R\sqrt{v_2}}{2K}
 \cos(2\pi t_2)\right)+O\left(\frac{R^{\varepsilon}}{K}
\right).
\end{split}
\end{equation}

Applying \eqref{eq:D}, \eqref{eq:C}, and \eqref{eq:taylorapplied}, in \eqref{eq:B}, and recalling $\lVert \alpha \lVert=m \lVert \delta \rVert \le \lVert \beta_b\rVert= |n|\lVert \delta \rVert$ and that we weight terms with $m \neq |n|$ by a factor of $2$, we deduce that
\begin{equation} \label{eq:E}
\begin{split}
\tmop{Var}(R,K)
={}&
\frac{\pi^2R^4}{K^2}
\sum_{\substack{m\in\mathbb N,\ n\in\mathbb Z\\
1\le m\le |n|\le \frac{K^{1+\varepsilon}}{R}}}
\left(2-1_{m=|n|}\right)
\sumstar_{\substack{\delta\in\mathbb Z^2\\
1\le \lVert\delta\rVert\le \frac{K^{1+\varepsilon}}{|n|R}}}
\int_{\mathbb R^2}W(v_1)W(v_2)
\\
&\quad\times
\int\limits_{\substack{t \in[-1/2,1/2] \\ \lVert t \rVert_{\frac12 \mathbb Z} \le \frac{1}{R^{1-\varepsilon}m \lVert \delta \rVert}}}
\bigg(e\left(R\lVert\delta\rVert
\left(n\sqrt{v_2}-m\sqrt{v_1}\right)
\sin(2\pi t)\right)
\\
&\quad\times
|\widehat f|^2\left(
\frac{\pi nR\lVert\delta\rVert\sqrt{v_2}}{2K}
\cos(2\pi t)\right)
+O\left(\frac{R^{\varepsilon}}{K}\right)\bigg)
\,dt\,dv_1\,dv_2
+O(R^{-A})
\end{split}
\end{equation}
for any $A>0$.

We now bound the error term appearing in the integrand by
\begin{equation} \label{eq:F}
\ll
\frac{R^4}{K^2}\cdot\frac{R^\varepsilon}{K}
\sum_{\substack{m\in\mathbb N,\ n\in\mathbb Z\\
1\le m\le |n|\le \frac{K^{1+\varepsilon}}{R}}} \,\,
\sumstar_{1 \le \lVert\delta\rVert\le \frac{K^{1+\varepsilon}}{|n|R}}
\frac{1}{mR\lVert\delta\rVert}
\ll
\frac{R^{2+\varepsilon}}{K^2}.
\end{equation}
Combining \eqref{eq:E} and \eqref{eq:F} and recalling the definition of
$\Sigma_{K,R}(v_1,v_2)$ given in \eqref{eq:sigmadef} completes the proof.
\end{proof}

\subsection{Simplifying the lattice point sum}
Let \[ V(u)= 2u W(u^2) 1_{(0,\infty)}(u).\] We next establish a cleaner expression for the main term appearing in Lemma \ref{lem:poissonagain}.

\begin{lemma} \label{lem:clean}
Let $\eta>0$. For $R^{1+\eta}<K< R^{2-\eta}$,
we have for any $\varepsilon>0$ that
\begin{equation} \label{eq:clean}
\begin{split}
&\int_{\mathbb R^2} W(v_1) W(v_2) \Sigma_{K,R}(v_1,v_2) \, dv_1 dv_2 \\
&= \frac{1}{\pi R} \sum_{m,n \in \mathbb N} \, \sumstar_{\delta \in \mathbb Z^2 \setminus\{0\}} \frac{1}{n\lVert \delta \rVert}  \int_{\mathbb R} V(t) V(\tfrac{m}{n} t) \, |\widehat f|^2\bigg(\frac{\pi m R \lVert \delta \rVert}{2K}t \bigg)\, dt+O\bigg( \frac{K^{2+\varepsilon}}{R^5} \bigg).
\end{split}
\end{equation}
\end{lemma}

\begin{proof}
Write $Q=\frac{\pi n R \lVert \delta \rVert}{2K}$ so $|Q| \le R^{\varepsilon}$.
Using Taylor expansions, we see that
\begin{equation} \label{eq:taylorp1}
\begin{split}
&\int_{|t| \le \frac{1}{R^{1-\varepsilon}m \lVert \delta \rVert}} e(R\lVert \delta \rVert (n \sqrt{v_2}-m\sqrt{v_1}) \sin(2\pi t))
|\widehat f|^2(Q \sqrt{v_2} \cos(2\pi t)) dt\\
& \qquad =\int_{|t|\le  \frac{1}{R^{1-\varepsilon}m \lVert \delta \rVert}}
e(2\pi R \lVert \delta \rVert (n \sqrt{v_2}-m \sqrt{v_1}) t) dt
|\widehat f|^2 (Q \sqrt{v_2})\\
&\qquad \qquad  +O\bigg( \frac{K^{1+\varepsilon}}{R^4m^4\lVert \delta \rVert^4}+\frac{|n| R^{\varepsilon}}{KR^2\lVert \delta \rVert^2 m^3}\bigg).
\end{split}
\end{equation}
Similarly, we have that
\begin{equation} \label{eq:taylorp2}
\begin{split}
&\int_{|t-\frac12| \le \frac{1}{R^{1-\varepsilon}m \lVert \delta \rVert}} e(R\lVert \delta \rVert (n \sqrt{v_2}-m\sqrt{v_1}) \sin(2\pi t))
|\widehat f|^2(Q \sqrt{v_2} \cos(2\pi t)) dt\\
& \qquad =\int_{|s|\le \frac{1}{R^{1-\varepsilon}m \lVert \delta \rVert}}
e(-2\pi R \lVert \delta \rVert (n \sqrt{v_2}-m \sqrt{v_1}) s) ds
|\widehat f|^2 (-Q \sqrt{v_2})\\
&\qquad \qquad  +O\bigg( \frac{K^{1+\varepsilon}}{R^4m^4\lVert \delta \rVert^4}+\frac{|n| R^{\varepsilon}}{KR^2\lVert \delta \rVert^2  m^3}\bigg).
\end{split}
\end{equation}

Since $|\widehat f|^2$ is even, making the change of variables $t=-s$ in \eqref{eq:taylorp2} shows that the main terms on the right-hand sides of \eqref{eq:taylorp1} and \eqref{eq:taylorp2} are equal.
Let $h_{Q}(t)=|\widehat f|^2(Qt)$ and
\[
\mathcal I (\delta,m,n)=\int_{\mathbb R^2} W(v_1)W(v_2) h_{  Q}( \sqrt{v_2})\int_{|t| \le \frac{1}{R^{1-\varepsilon}m\lVert \delta \rVert}} e(  2\pi R \lVert \delta \rVert (n \sqrt{v_2}-m \sqrt{v_1}) t) \, dt dv_1 dv_2.
\]
Applying \eqref{eq:taylorp1} and \eqref{eq:taylorp2} we get that
\begin{equation} \label{eq:taylorp3}
\begin{split}
&\int_{\mathbb R^2}W(v_1)W(v_2)
\Sigma_{K,R}(v_1,v_2)\,dv_1\,dv_2\\
&=2 \sum_{\substack{m\in\mathbb N,\ n\in\mathbb Z\\
1\le m\le |n|\le \frac{K^{1+\varepsilon}}{R}}}
\left(2-1_{m=|n|}\right)
\sumstar_{\substack{\delta\in\mathbb Z^2\\
1\le \lVert\delta\rVert\le \frac{K^{1+\varepsilon}}{|n|R}}}  \mathcal I(\delta,m,n)+O\bigg( \frac{K^{2+\varepsilon}}{R^5}\bigg).
\end{split}
\end{equation}
Here we estimated the error term as
\[
\ll R^{\varepsilon} \bigg( \sum_{\substack{m\in\mathbb N,\ n\in\mathbb Z\\
1\le m\le |n|\le \frac{K^{1+\varepsilon}}{R}}}
\sumstar_{\substack{\delta\in\mathbb Z^2\\
1\le \lVert\delta\rVert\le \frac{K^{1+\varepsilon}}{|n|R}}} \bigg( \frac{K}{R^4m^4\lVert \delta \rVert^4}+\frac{|n| }{KR^2\lVert \delta \rVert^2 m^3} \bigg)\bigg)\ll \frac{K^{2+\varepsilon}
}{R^5}
\]

We turn to evaluating $\mathcal I(\delta,m,n)$, and make the substitutions $u_1=\sqrt{v_1}$, $u_2=\sqrt{v_2}$, and $x=2\pi R \lVert \delta \rVert t$ to get that
\[
\mathcal I(\delta,m,n)
= \int_{\mathbb R^2} V(u_1)V(u_2) h_{ Q}(u_2)
\int_{|x| \le \frac{2\pi R^{\varepsilon}}{m}} e(  x (nu_2-mu_1)) \frac{dx}{2\pi R \lVert \delta \rVert} du_1 du_2.
\]
By repeatedly integrating by parts, with respect to $u_1$ we can extend the integral over $x$ to $|x| < R^{\varepsilon}$ at the cost of an error term of size $O(R^{-A})$ for any $A>0$. Hence, writing $t$ in place of $x$ for the variable of integration, we get that
\begin{equation} \label{eq:fourier}
\begin{split}
\mathcal I(\delta,m,n)=&\int_{-R^{\varepsilon}}^{R^{\varepsilon}} \int_{\mathbb R} V(u_1) e(  - m tu_1) du_1 \int_{\mathbb R} V(u_2)h_{  Q}( u_2)e(   tnu_2) du_2 \frac{dt}{2 \pi R \lVert \delta \rVert}+O(R^{-A})\\
&=\int_{-R^{\varepsilon}}^{R^{\varepsilon}} \widehat V(  mt) (\widehat V \ast \widehat h_{  Q})(  -tn) \, \frac{dt}{2\pi R \lVert \delta \rVert}+O(R^{-A}) \\
&=\int_{\mathbb R} \widehat V(  mt) \widehat{(V \cdot h_{  Q})}(  -tn) \frac{dt}{2\pi R \lVert \delta \rVert}+O(R^{-A}).
\end{split}
\end{equation}
Using \eqref{eq:fourier} and applying the Plancherel theorem we obtain
\begin{equation} \label{eq:plancherelapplied}
\begin{split}
\mathcal I(\delta,m,n)= &\int_{\mathbb R}  V(t) V(\tfrac{m}{n}t) h_{  Q}\left(\frac{m}{n} t \right) \frac{dt}{2\pi R \lVert \delta \rVert n}+O(R^{-A}) \\
=&
 \int_{\mathbb R} V(t) V(\tfrac{m}{n}t) |\widehat f|^2\bigg( \frac{   \pi mR \lVert \delta \rVert}{2K}t\bigg) \, \frac{dt}{2\pi R \lVert \delta \rVert n}+O(R^{-A}).
\end{split}
\end{equation}
To complete the proof first recall that $V$ is supported on the positive reals so that $V(\frac{m}{n}t)V(t)$ restricts $n$ to be positive in the main term in \eqref{eq:plancherelapplied}. Now apply \eqref{eq:plancherelapplied} in \eqref{eq:taylorp3}, and also use the rapid decay of $|\widehat f|^2$ to extend the sums to all $n$ and $\delta$ up to a negligible error term of size $O(R^{-A})$, noting that $m \asymp n$ in the main term in \eqref{eq:plancherelapplied} since $V$ is compactly supported on the positive reals. Finally, by a change of variable in the integral and by symmetry, it is easy to see that we can sum over $m,n\in \mathbb{N}$ and remove the weight $2-1_{m=|n|}$.
 \end{proof}

\section{Evaluation  of the main term} \label{sec:mainterm}

Before proceeding to the evaluation of the main term appearing in \eqref{eq:clean}, we require the following simple lemma. Recall that
\[
 V(u)= 2u W(u^2) 1_{(0,\infty)}(u).
\]

\begin{lemma} \label{lem:euler} Let $m \in \mathbb N$ and suppose $t$ is in the support of $V$. Then for any $A>0$ we have that
\begin{equation} \label{eq:coeff}
\sum_{n \ge 1} \frac{1}{n} V\left(\frac{m }{n} t \right)
= \int_0^\infty V(u) \frac{du}{u}+O(m^{-A}),
\end{equation}
where the implied constant depends on at most $A$, and $V$.
\end{lemma}
\begin{proof}
Define the function $g$ by
\[
g(u) = \frac{1}{u}V\left(\frac{1}{u}\right).
\]
Since $g$ is smooth and supported on the positive real line, the Poisson summation formula and the rapid decay of $\widehat{g}$ give
\begin{equation} \label{eq:euler1}
\sum_{n \ge 1} \frac{1}{n} V\left(\frac{m }{n} t \right) = \frac{1}{mt} \sum_{n \in \mathbb{Z}} g\left(\frac{n}{mt}\right)=\int_0^{\infty} g(u) \, du
+O( m^{-A})
\end{equation}
for any $A>0$.
We complete the proof by making the change of variable $1/u \rightarrow u $ to see that
\begin{equation} \label{eq:euler3}
\int_0^{\infty} g(u) du=\int_0^\infty V(u) \frac{du}{u}.
\end{equation} \end{proof}

In the next lemma we evaluate the main term in Lemma \ref{lem:clean} using a contour integration argument.

\begin{lemma}  \label{lem:mainterm}
Let $\eta>0$.  For $R^{1+\eta}<K < R^{2-\eta}$, we have for any $\varepsilon>0$ that
\begin{equation} \label{eq:mainterm1}
\sum_{m,n \in \mathbb N } \, \sumstar_{\delta \in \mathbb Z^2 \setminus\{0\}} \frac{1 }{n\lVert \delta \rVert}  \int_{\mathbb R} V(t) V(\tfrac{m}{n} t) \, |\widehat f|^2\bigg(\frac{\pi m R \lVert \delta \rVert}{2K}t \bigg)\, dt=
\frac{C_1}{\pi} \frac{K}{R} \log\frac{K}{R}+\frac{C_0}{\pi} \frac{K}{R} +O \left( \left( \frac{K}{R}\right)^{\varepsilon} \right),
\end{equation}
where $C_{1},C_{0}$ are the constants defined in (\ref{eq:main_consts}).
\end{lemma}
\begin{proof}
Recall the following well-known formula for the Epstein zeta-function. For $s=\sigma+i\tau$ with $\sigma>1$ we have that
\begin{equation} \label{eq:epstein}
E(s):=\sumstar_{\delta \in \mathbb Z^2 \setminus\{0\}} \frac{1}{\lVert \delta \rVert^{2s}}=\frac{4 \zeta(s)L(s,\chi_{-4})}{\zeta(2s)},
\end{equation}
where $\zeta(s)$ denotes the Riemann zeta-function,
and $L(s,\chi_{-4})$ is the Dirichlet $L$-function attached to the non-principal character $\pmod 4$, which we write as $\chi_{-4}$. That is,
 $\chi_{-4}:  \mathbb Z \rightarrow \{0,\pm 1\}$ with $\chi_{-4}(n)=(-1)^{(n-1)/2}$ if $n$ is odd and $\chi_{-4}(n)=0$ otherwise, and for $\sigma>1$
\[
L(s,\chi_{-4})=\sum_{m =1}^{\infty} \frac{\chi_{-4}(m)}{m^s}.
\]
Note that $L(s,\chi_{-4})$ admits an analytic continuation to $\mathbb C$.
It follows that $E(s)$ is holomorphic for $\sigma \ge \frac12$, except for a simple pole at $s=1$. Using the following expansions near $s=1$
\begin{equation} \label{eq:zeta}
\zeta(s)=\frac{1}{s-1}+\gamma+O(|s-1|),
\end{equation}
where $\gamma=0.57721\ldots$ is Euler's constant,
\[
L(s,\chi_{-4})=\frac{\pi}{4}+L'(1,\chi_{-4})(s-1)+O(|s-1|^2),
\]
where we used $L(1,\chi_{-4})=\frac{\pi}{4}$, and
\[
\frac{1}{\zeta(2s)}=\frac{6}{\pi^2}\bigg(1-2\frac{\zeta'}{\zeta}(2)(s-1)+O(|s-1|^2) \bigg)
\]
it is not hard to show that near $s=1$ one has
\begin{equation} \label{eq:Enear1}
E(s)=\frac{6}{\pi(s-1)}+\frac{6}{\pi} \bigg(\gamma+\frac{L'}{L}(1,\chi_{-4})-2\frac{\zeta'}{\zeta}(2)\bigg)+O(|s-1|).
\end{equation}

Also, note that by Lemma \ref{lem:euler} for $m \in \mathbb N$ and $t$ in the support of $V$ we have that
\[
a_{V,t}(m):=\sum_{n \ge 1} \frac{1}{n} V\bigg(\frac{m}{n}t \bigg) = \int_0^\infty V(u) \frac{du}{u}+O(m^{-A})
\]
for any $A>0$.
Hence, it follows that
\begin{equation} \label{eq:Ddef}
D_{V,t}(s):=\sum_{m=1}^{\infty} \frac{a_{V,t}(m)}{m^s}= \int_0^\infty V(u) \frac{du}{u} \zeta(s)+\mathcal R_t(s)
\end{equation}
where $\mathcal R_t(s)$ is an entire function, and satisfies the bound $|\mathcal R_t(s)|=O(1)$ in any fixed strip $A\le \sigma \le B$. In particular, $D_{V,t}(s)$ is holomorphic in $\mathbb{C}$ except for a simple pole at $s=1$ and, using \eqref{eq:zeta}, near $s=1$ we have that
\begin{equation} \label{eq:Dnear1}
D_{V,t}(s)=\frac{\int_0^\infty V(u) \frac{du}{u}}{s-1}+\gamma \int_0^\infty V(u) \frac{du}{u}+\mathcal R_t(1)+O(|s-1|).
\end{equation}

Write $g=|\widehat f|^2$ and $\widetilde g(s) = \int_{0}^{\infty}x^{s-1}g\left(x\right)dx$ for the Mellin transform of $g$ which is holomorphic  for $\sigma>0$. Applying the Mellin inversion theorem we see that
\begin{equation} \label{eq:mellin}
\sum_{m,n \in \mathbb N} \frac{V(\frac{m}{n} t)
}{n}\sumstar_{\delta \in \mathbb Z^2 \setminus\{0\}} \frac{1}{\lVert \delta \rVert} g\bigg(\frac{\pi mR \lVert \delta \rVert}{2K} t \bigg) =\frac{1}{2\pi i} \int_{(2)} \bigg(\frac{2K}{\pi R t} \bigg)^s E\left( \frac{s+1}{2}\right) D_{V,t}(s) \widetilde g(s) \, ds.
\end{equation}
Using the rapid decay of $\widetilde g(s)$ in $\tau$ in any fixed strip $0<A\le \sigma \le B$ and recalling the formulas \eqref{eq:epstein} and \eqref{eq:Ddef}, we may shift the contour on the right-hand side of \eqref{eq:mellin} to $\sigma=\varepsilon$. Picking up a double pole at $s=1$ we obtain
\begin{align} \label{eq:contourshift}
\frac{1}{2\pi i} \int_{(2)} \bigg(\frac{2K}{\pi R t} \bigg)^s E\left( \frac{s+1}{2}\right) D_{V,t}(s) \widetilde g(s) \, ds=&
\tmop{Res}_{s=1} \bigg(\bigg(  \frac{2K}{\pi R t}\bigg)^s E\left(\frac{s+1}{2} \right) D_{V,t}(s) \widetilde g(s) \bigg) \nonumber \\
&+\frac{1}{2\pi i} \int_{(\varepsilon)} \bigg(\frac{2K}{\pi R t} \bigg)^s E\left( \frac{s+1}{2}\right) D_{V,t}(s) \widetilde g(s) \, ds.
\end{align}
Using the rapid decay of $\widetilde g(s)$ in $\tau$ on the line $\sigma=\varepsilon$ we conclude for $t$ in the support of $V$ that
\begin{equation} \label{eq:contourbd}
\frac{1}{2\pi i} \int_{(\varepsilon)} \bigg(\frac{2K}{\pi R t} \bigg)^s E\left( \frac{s+1}{2}\right) D_{V,t}(s) \widetilde g(s) \, ds \ll \left( \frac{K}{R} \right)^{\varepsilon}.
\end{equation}

To evaluate the residue on the right-hand side of \eqref{eq:contourshift} write $X=\frac{2K}{\pi R t}$ and note that near $s=1$ we have
\begin{equation} \label{eq:Xgnear1}
X^s \widetilde g(s)=X \widetilde g(1)+X(\widetilde g'(1)+\widetilde g(1) \log X)(s-1)+O(|s-1|^2),
\end{equation}
where we have allowed the implied constant to depend on $X$ since we take $s \rightarrow 1$.
Hence, using \eqref{eq:Enear1}, \eqref{eq:Dnear1}, and \eqref{eq:Xgnear1} we get that
\begin{align} \label{eq:res-eval}
&\tmop{Res}_{s=1} \bigg(X^s E\left(\frac{s+1}{2} \right) D_{V,t}(s) \widetilde g(s)\bigg) \nonumber \\
&= \frac{12 X}{\pi} \int_0^\infty V(u) \frac{du}{u}
\bigg(\widetilde g(1)\bigg(\log X+\frac{3\gamma}{2}+\frac12\frac{L'}{L}(1,\chi_{-4})-\frac{\zeta'}{\zeta}(2) \bigg)+\widetilde g'(1) \bigg)
+\frac{12X}{\pi} \mathcal R_t(1)\widetilde g(1).
\end{align}
Using \eqref{eq:contourbd} and \eqref{eq:res-eval} in \eqref{eq:contourshift}, and then applying the resulting formula in \eqref{eq:mellin} and integrating against $V(t)$ over $t$ in $\mathbb R$ completes the proof of \eqref{eq:mainterm1} up to evaluating the constants.

To evaluate the constants, we first note that since $|\widehat f|^2$ is even that
\begin{equation} \label{eq:int1}
\widetilde g(1)=\int_0^{\infty}|\widehat f|^2(x) \, dx=\frac{1}{2} \int_{\mathbb R}  |\widehat f|^2(x) \, dx = \frac{1}{2} \int_{\mathbb R} |f(t)|^2 \, dt.
\end{equation}
where the last step follows from the Plancherel theorem.
Hence, recalling $X=\frac{2K}{\pi R t}$, $V(u)=2uW(u^2)$, and integrating the right-hand side against $V(t)$ over $t \in \mathbb R$, it follows from  \eqref{eq:int1}  that the leading constant is equal to
\[
\frac{12}{\pi } \cdot \frac{2}{\pi }\cdot\frac{1}{2} \int_{\mathbb R} |f(t)|^2 \, dt\cdot  \bigg(\int_0^{\infty} V(u) \frac{du}{u} \bigg)^2=\frac{12}{\pi^2} \, \int_{\mathbb R} |f(t)|^2 \, dt \, \bigg(\int_0^{\infty} V(u) \frac{du}{u} \bigg)^2 = \frac{C_1}{\pi}.
\]

As for the constant of the lower order term $\frac{K}{R},$ by \eqref{eq:mellin}, \eqref{eq:contourshift} and \eqref{eq:res-eval}, it is enough to show that
\[
\int_{\mathbb{R}}\mathcal{R}_{t}\left(1\right)\frac{V\left(t\right)}{t}dt =2c_{4}.
\]
To prove this, fix $s$ with $\sigma>1$ and note that the Mellin inversion theorem can be applied to $V$ on the $-\varepsilon$ line to get that
\[
\sum_{n \ge 1} \frac{V(\frac{m}{n}t)}{n}=\frac{1}{2\pi i} \int_{(-\varepsilon)}
(mt)^{-w} \zeta(1-w)\widetilde V(w) \, dw.
\]
By shifting the contour of the
integral to the right (and thus picking the residue at $w=0$ with a minus sign), we have
\begin{equation} \label{eq:Dform}
\begin{split}
D_{V,t}(s)& =\frac{1}{2\pi i}\int_{\left(-\varepsilon\right)}\zeta\left(1-w\right)\zeta\left(s+w\right)\widetilde{V}\left(w\right)t^{-w}dw\\
 & =\frac{1}{2\pi i}\int_{\left(4\right)}\zeta\left(1-w\right)\zeta\left(s+w\right)\widetilde{V}\left(w\right)t^{-w}dw+\zeta\left(s\right)\widetilde{V}\left(0\right).
\end{split}
\end{equation}

Thus, by \eqref{eq:Ddef} and \eqref{eq:Dform} we get that
\[
\mathcal{R}_{t}\left(s\right)=D_{V,t}(s)-\zeta\left(s\right)\widetilde{V}\left(0\right)=\frac{1}{2\pi i}\int_{\left(4\right)}\zeta\left(1-w\right)\zeta\left(s+w\right)\widetilde{V}\left(w\right)t^{-w}dw
\]
so that
\begin{align*}
\int_{\mathbb{R}}\mathcal{R}_{t}\left(1\right)\frac{V\left(t\right)}{t}dt & =\frac{1}{2\pi i}\int_{\left(4\right)}\zeta\left(1-w\right)\zeta\left(1+w\right)\widetilde{V}\left(-w\right)\widetilde{V}\left(w\right)dw\\
 & =2\cdot\frac{1}{2\pi i}\int_{\left(2\right)}\zeta\left(1-2w\right)\zeta\left(1+2w\right)\widetilde{W}\left(\frac12 -w\right)\widetilde{W}\left(\frac12 +w\right)dw\\
 & =2c_{4}
\end{align*}
 where we used the identity $\widetilde{V}\left(w\right)=\widetilde{W}\left(\frac{1+w}{2}\right)$.
\end{proof}

\section{ The proof of Theorem \ref{thm:mainthm}, part (2)} \label{sec:proof}
\begin{proof}[Proof of Theorem \ref{thm:mainthm}, part (2)]
Applying Lemmas  \ref{lem:poissonagain} and \ref{lem:clean}, we obtain
\begin{equation} \label{eq:varclean}
\tmop{Var}(R,K)=
\frac{\pi R^3}{K^2}  \sum_{m,n \in \mathbb N } \, \sumstar_{\delta \in \mathbb Z^2 \setminus\{0\}} \frac{1 }{n\lVert \delta \rVert}  \int_{\mathbb R} V(t) V(\tfrac{m}{n} t) \, |\widehat f|^2\bigg(\frac{\pi m R \lVert \delta \rVert}{2K}t \bigg)\, dt
+O\bigg(\frac{R^{2+\varepsilon}}{K^2}+ \frac{1}{R^{1-\varepsilon}} \bigg).
\end{equation}
To evaluate the main term on the right-hand side of \eqref{eq:varclean} we apply Lemma \ref{lem:mainterm} to see that
\begin{equation}
\begin{split}
\tmop{Var}(R,K)=&\frac{\pi R^3}{K^2}\bigg(\frac{C_1}{\pi} \frac{K}{R} \log \frac{K}{R}+\frac{C_0}{\pi} \frac{K}{R}+ O \left( \left( \frac{K}{R}\right)^{\varepsilon} \right) \bigg)
+O\bigg(\frac{R^{2+\varepsilon}}{K^2}+ \frac{1}{R^{1-\varepsilon}} \bigg)\\
=& \frac{R^2}{K} \bigg(C_1 \log\frac{K}{R}+ C_0 +O\bigg(\frac{R^{1+\varepsilon}}{K}\bigg)\bigg),
\end{split}
\end{equation}
which completes the proof. \end{proof}

\section{Very narrow sectors}
\label{sec:VeryNarrowSec}

In this section we will prove the third part of Theorem \ref{thm:mainthm}, which deals with very narrow sectors. We will need several  preparatory lemmas.
\begin{lemma}
\label{lem:VarExp}Let $\eta>0$.
For $K>R^{2+\eta}$, we have that
\[
\textup{Var}\left(R,K\right)=\frac{\pi}{3}\cdot\frac{c_{1}}{K}\sum_{\substack{\alpha,\beta\in\mathbb{Z}^{2}\backslash\left\{ 0\right\} \\
\theta_{\alpha},\theta_{\beta}\in[0,\pi/2)\\
\theta_{\alpha}=\theta_{\beta}
}
}W\left(\frac{\left\Vert \alpha\right\Vert ^{2}}{R^{2}}\right)W\left(\frac{\left\Vert \beta\right\Vert ^{2}}{R^{2}}\right)+O\left( \frac{R^{4}}{K^{2}}\right),
\]
where $c_{1}$ is the constant defined in (\ref{eq:pre_consts}).
\end{lemma}

\begin{proof}
We write
\begin{align}
\textup{Var}\left(R,K\right) & =\int_{0}^{\pi/2}\left|\sum_{\alpha\in\mathbb{Z}^{2}\backslash\left\{ 0\right\} }F_{K}\left(\theta-\theta_{\alpha}\right)W\left(\frac{\left\Vert \alpha\right\Vert ^{2}}{R^{2}}\right)\right|^{2}\frac{d\theta}{\pi/2}-\left(\frac{\widehat{f}\left(0\right)}{K}\sum_{\alpha\in\mathbb{Z}^{2}\setminus\{0\}}W\left(\frac{\left\Vert \alpha\right\Vert ^{2}}{R^{2}}\right)\right)^{2}.\label{eq:Var_expansion}
\end{align}
Clearly,
\begin{equation}
\frac{\widehat{f}\left(0\right)}{K}\sum_{\alpha\in\mathbb{Z}^{2}\setminus\{0\}}W\left(\frac{\left\Vert \alpha\right\Vert ^{2}}{R^{2}}\right) = O\left( \frac{R^{2}}{K}\right).\label{eq:Gauss}
\end{equation}
As for the second moment, we have
\begin{alignat}{1}
\int_{0}^{\pi/2}\left|\sum_{\alpha\in\mathbb{Z}^{2}\backslash\left\{ 0\right\} }F_{K}\left(\theta-\theta_{\alpha}\right)W\left(\frac{\left\Vert \alpha\right\Vert ^{2}}{R^{2}}\right)\right|^{2} \frac{d\theta}{\pi/2} & =16\sum_{\substack{\alpha,\beta\in\mathbb{Z}^{2}\backslash\left\{ 0\right\} \\
\theta_{\alpha},\theta_{\beta}\in[0,\pi/2)
}
}W\left(\frac{\left\Vert \alpha\right\Vert ^{2}}{R^{2}}\right)W\left(\frac{\left\Vert \beta\right\Vert ^{2}}{R^{2}}\right)\label{eq:angles_expansion}\\
 & \times\int_{0}^{\pi/2}F_{K}\left(\theta-\theta_{\alpha}\right)F_{K}\left(\theta-\theta_{\beta}\right)\frac{d\theta}{\pi/2}.\nonumber
\end{alignat}

If $\theta_{\alpha}\ne\theta_{\beta}$, then as in \eqref{angles_ineq}, we have
\[
\left|\theta_{\beta}-\theta_{\alpha}\right|\ge\left|\sin\left(\theta_{\beta}-\theta_{\alpha}\right)\right|\ge\frac{1}{\left\Vert \alpha\right\Vert \left\Vert \beta\right\Vert }\gg R^{-2},
\]
and hence for any $\theta\in[0,\pi/2)$ either $\left|\theta-\theta_{\alpha}\right|\gg R^{-2}$
or $\left|\theta-\theta_{\beta}\right|\gg R^{-2}$. By the assumption $K>R^{2+\eta}$ and the rapid decay
of $f$, we then have for $\theta_{\alpha}\ne\theta_{\beta}$ that
\begin{equation}
\int_{0}^{\pi/2}F_{K}\left(\theta-\theta_{\alpha}\right)F_{K}\left(\theta-\theta_{\beta}\right)\frac{d\theta}{\pi/2}=
O\left(K^{-A}\right)\label{eq:off_diag}
\end{equation}
for any $A>0$.
If $\theta_{\alpha}=\theta_{\beta},$ then by Parseval's
theorem and the Poisson summation formula applied to the smooth function $\frac{1}{K} \left| \widehat{f}(\frac k K)\right|^{2}$
\begin{align}
\int_{0}^{\pi/2}\left|F_{K}\left(\theta\right)\right|^{2}\frac{d\theta}{\pi/2} & =\sum_{k\in\mathbb{Z}}\left|\widehat{F}_K\left(k\right)\right|^{2}=\frac{1}{K^{2}}\sum_{k\in\mathbb{Z}}\left|\widehat{f}\left(\frac{k}{K}\right)\right|^{2}=\frac{1}{K}\int_{\mathbb{R}}\left|\widehat{f}\left(x\right)\right|^{2}\,dx+O\left(K^{-A}\right)\label{eq:diag}\\
 & =\frac{1}{K}\int_{\mathbb{R}}\left|f\left(x\right)\right|^{2}\,dx+O\left(K^{-A}\right).\nonumber
\end{align}
The claim follows from combining (\ref{eq:Var_expansion}), (\ref{eq:Gauss}), (\ref{eq:angles_expansion}), (\ref{eq:off_diag}),
and (\ref{eq:diag}).
\end{proof}
\begin{lemma}
\label{lem:WWSum}We have that
\[
 \sum_{\substack{\alpha,\beta\in\mathbb{Z}^{2}\backslash\left\{ 0\right\} \\
\theta_{\alpha},\theta_{\beta}\in[0,\pi/2)\\
\theta_{\alpha}=\theta_{\beta}
}
}W\left(\frac{\left\Vert \alpha\right\Vert ^{2}}{R^{2}}\right)W\left(\frac{\left\Vert \beta\right\Vert ^{2}}{R^{2}}\right)= I\left(R\right),
\]
with
\begin{equation} \label{eq:I2def}
I\left(R\right)= \frac14 \cdot \frac{1}{2\pi i}\int_{\left(2\right)}\frac{1}{2\pi i}\int_{\left(2\right)}R^{2\left(s+s'\right)}\widetilde{W}\left(s\right)\widetilde{W}\left(s'\right)\zeta\left(2s\right)\zeta\left(2s'\right)E\left(s+s'\right)dsds'
\end{equation}
where $\widetilde{W}$ is the Mellin transform of $W$ and $E(s)$ is the Epstein zeta-function \eqref{eq:epstein}.
\end{lemma}

\begin{proof}
Write
\begin{equation}
\sum_{\substack{\alpha,\beta\in\mathbb{Z}^{2}\backslash\left\{ 0\right\} \\
\theta_{\alpha},\theta_{\beta}\in[0,\pi/2)\\
\theta_{\alpha}=\theta_{\beta}
}
}W\left(\frac{\left\Vert \alpha\right\Vert ^{2}}{R^{2}}\right)W\left(\frac{\left\Vert \beta\right\Vert ^{2}}{R^{2}}\right)= \frac 14 \sideset{}{^{*}}\sum_{\substack{\alpha\in\mathbb{Z}^{2}\backslash\left\{ 0\right\} }
}\left(\sum_{n=1}^{\infty}W\left(\frac{n^{2}\left\Vert \alpha\right\Vert ^{2}}{R^{2}}\right)\right)^{2}\label{eq:WWsum}
\end{equation}
where the sum $\sideset{}{^{*}}\sum$ is taken over primitive lattice
points. By the Mellin inversion theorem, we have
\begin{equation}
W\left(\frac{n^{2}\left\Vert \alpha\right\Vert ^{2}}{R^{2}}\right)=\frac{1}{2\pi i}\int_{\left(2\right)}\frac{\widetilde{W}\left(s\right)R^{2s}}{n^{2s}\left\Vert \alpha\right\Vert ^{2s}}ds,\label{eq:W_Mellin}
\end{equation}
and the claim follows upon substituting (\ref{eq:W_Mellin}) in (\ref{eq:WWsum})
and using the formula \eqref{eq:epstein}.
\end{proof}
\begin{lemma}
\label{lem:I_asymp}Let $I(R)$ be as in \eqref{eq:I2def}.
We have that
\[
I\left(R\right)=\frac{3}{\pi}c_{2}R^{2}\log R+\frac{3}{\pi}c_{6}R^{2}+O\left(R\right)
\]
where $c_{2},c_{6}$ are the constants defined in (\ref{eq:pre_consts}).
\end{lemma}

\begin{proof}
We have that
\begin{equation}
I\left(R\right)=\frac{1}{2\pi i}\int_{\left(2\right)}R^{2s}\widetilde{W}\left(s\right)\zeta\left(2s\right)\left(\frac{1}{2\pi i}\int_{\left(2\right)}F_{s}\left(s'\right)ds'\right)ds\label{eq:I_def}
\end{equation}
where
\[
F_{s}\left(s'\right)=\frac14R^{2s'}\widetilde{W}\left(s'\right)\zeta\left(2s'\right)E(s+s').
\]
Consider the inner integral in (\ref{eq:I_def}). By the rapid decay
of $\widetilde{W}(\sigma+i\tau)$ in the variable $\tau$
in any fixed strip, we may shift the integral to $\sigma'=-3/2$, and pick up the residues from simple poles
at $s'=1/2$ and $s'=1-s$. It follows that
\begin{equation} \label{eq:Fres1}
\frac{1}{2\pi i}\int_{\left(2\right)}F_{s}\left(s'\right)ds'=\textup{Res}_{s'=1/2}F_{s}\left(s'\right)+\textup{Res}_{s'=1-s}F_{s}\left(s'\right)+O\left(R^{-3}\left(1+|s|\right)^{d}\right)
\end{equation}
 for some $d>0$. A simple calculation gives
\begin{equation} \label{eq:Fres2}
\textup{Res}_{s'=1-s}F_{s}\left(s'\right)=\frac{3}{2\pi}R^{2-2s}\widetilde{W}\left(1-s\right)\zeta\left(2-2s\right)
\end{equation}
and
\begin{equation} \label{eq:Fres3}
\textup{Res}_{s'=1/2}F_{s}\left(s'\right)=\frac18 R\,\widetilde{W}\left(\frac12 \right)E\left(s+\tfrac12\right).
\end{equation}

Substituting \eqref{eq:Fres2} and \eqref{eq:Fres3} in \eqref{eq:Fres1} then applying the resulting formula in (\ref{eq:I_def}) and shifting the contour, we get that
\begin{equation}
I\left(R\right)=\frac{\widetilde{W}\left(\frac12\right)}{8}R\cdot\frac{1}{2\pi i}\int_{\left(2\right)}G\left(s\right)ds+\frac{3}{\pi}\frac{c_{4}}{2}R^{2}+O\left(R\right)\label{eq:I_outer}
\end{equation}
where
\[
G\left(s\right)= R^{2s}\widetilde{W}\left(s\right)\zeta\left(2s\right)E\left(s+\tfrac12\right).
\]
Shifting the contour and picking up the residue at $s=\frac12$, we get
\begin{equation}
\frac{1}{2\pi i}\int_{\left(2\right)}G\left(s\right)ds=\frac{1}{2\pi i}\int_{\left(0\right)}G\left(s\right)ds+\textup{Res}_{s=\frac12}G\left(s\right)=\textup{Res}_{s=\frac12}G\left(s\right)+O\left(1\right).\label{eq:G_int}
\end{equation}
Around $s=\frac12$ we have that
\begin{equation} \label{eq:Gres1}
\zeta\left(2s\right)=\frac{1}{2\left(s-\tfrac12\right)}+\gamma+O\left(|s-\tfrac12|\right)
\end{equation}
and by \eqref{eq:Enear1}
\begin{equation} \label{eq:Gres2}
E(s+\tfrac12)=\frac{6}{\pi(s-\frac12)}+\frac{6}{\pi} \bigg(\gamma+\frac{L'}{L}(1,\chi_{-4})-2\frac{\zeta'}{\zeta}(2)\bigg)+O(|s-\tfrac12|).
\end{equation}
Additionally, near $s=\frac12$
\begin{equation} \label{eq:Gres3}
R^{2s} \widetilde{W}(s)=R\widetilde{W}\left(\frac12\right)+(2R\log R\widetilde{W}\left(\frac12\right) + \widetilde{W}'\left(\frac12\right)R)\left(s-\tfrac12\right)+O\left(\left|s-\tfrac12\right|^{2}\right).
\end{equation}

Thus, we conclude using \eqref{eq:Gres1}, \eqref{eq:Gres2}, and \eqref{eq:Gres3} 
\begin{align}
\textup{Res}_{s=\frac12}G\left(s\right) & =\frac{6}{\pi} \widetilde{W}\left(\frac12\right) R\log R+\frac{6}{\pi}\left(\frac{2}{3} \widetilde{W}\left(\frac12\right) \left(\frac{9\gamma}{4}+\frac{3L'\left(1,\chi_{-4}\right)}{\pi}-\frac{9\zeta'\left(2\right)}{\pi^{2}}\right)+\frac{\widetilde{W}'\left(\frac12\right)}{2}\right)R.\label{eq:ResG}
\end{align}
The claim now follows upon noting that
\begin{align*}
\widetilde{W}\left(\frac12\right)&=\int_{0}^{\infty}\frac{1}{\sqrt{x}}W\left(x\right)dx=2\int_{0}^{\infty}W\left(u^{2}\right)\,du,
\\
\widetilde{W}'\left(\frac12\right)&=\int_{0}^{\infty}\frac{1}{\sqrt{x}}W\left(x\right) \log x dx=4\int_{0}^{\infty}W\left(u^{2}\right) \log u\,du,
\end{align*}
and combining (\ref{eq:I_outer}), (\ref{eq:G_int}),
and (\ref{eq:ResG}).
\end{proof}

\begin{proof}[Proof of Theorem \ref{thm:mainthm}, part (3)]

Assume that $K>R^{2+\eta}$. The proof of the third part of Theorem
\ref{thm:mainthm} follows by combining Lemma \ref{lem:VarExp}, Lemma
\ref{lem:WWSum} and Lemma \ref{lem:I_asymp}.
\end{proof}

\appendix

\section{\label{sec:Constants}Constants}

Denote by $\widetilde{W}$ the Mellin transform of $W$, $\widehat{f}$ the Fourier transform of $f$, $\zeta(s)$ the Riemann zeta-function,
$L(s,\chi_{-4})$  the Dirichlet $L$-function attached to the non-principal character $\pmod 4$, and let $\gamma=0.57721\ldots$ denote Euler's constant. The following constants are used throughout the paper.

\begin{align}
c_{1} & =\frac{48}{\pi}\int_{\mathbb{R}}\left|f\left(x\right)\right|^{2}\,dx,\nonumber \\
c_{2} & =\left(\int_{0}^{\infty}W\left(u^{2}\right)\,du\right)^{2},\nonumber \\
c_{3} & =\int_{0}^{\infty}W\left(u^{2}\right)\,du\int_{0}^{\infty}W\left(u^{2}\right)\log u\,du,\label{eq:pre_consts}\\
c_{4} & =\frac{1}{2\pi i}\int_{\left(2\right)}\widetilde{W}\left(\frac12 +s\right)\widetilde{W}\left(\frac12-s\right)\zeta\left(1+2s\right)\zeta\left(1-2s\right)ds,\nonumber \\
c_{5} & =c_{1}c_{2}\log\left(\frac{2}{\pi} \right)+\frac{48}{\pi}c_{2}\int_{\mathbb{R}}|\widehat{f}|^{2}\left(x\right)\log\left|x\right|\,dx-2c_{1}c_{3},\nonumber \\
c_{6} & =\frac{2}{3}c_{2}\left(\frac{9\gamma}{4}+\frac{3L'\left(1,\chi_{-4}\right)}{\pi}-\frac{9\zeta'\left(2\right)}{\pi^{2}}\right)+c_{3}+\frac{c_{4}}{2},\nonumber \end{align}
and
\begin{align}
C_{1} & =c_{1}c_{2}=\frac{48}{\pi}\int_{\mathbb{R}}\left|f\left(x\right)\right|^{2}\,dx\left(\int_{0}^{\infty}W\left(u^{2}\right)\,du\right)^{2},\label{eq:main_consts}\\
D_{0} & =c_{1}c_{6},\nonumber \\
C_{0} & =D_{0}+c_{5}.\nonumber
\end{align}

\section*{Acknowledgements}
We express our deep gratitude to Alexei Entin, Bingrong Huang, and Igor Wigman for stimulating and insightful discussions on various aspects of our work.

Gemini, ChatGPT, and Claude were used to assist with proofreading, background/literature review, and to typeset handwritten notes into LaTeX.

\end{document}